\documentclass[a4paper,11pt]{article}
\usepackage{amssymb,amsmath,amsthm,latexsym}
\usepackage{amsfonts}
	\usepackage{amsfonts}
\usepackage{graphicx}
\usepackage[pdftex,bookmarks,colorlinks=false]{hyperref}
\usepackage{verbatim}
\usepackage{caption}
\usepackage{subcaption}
\usepackage[hmargin=1.15in,vmargin=1.15in]{geometry}

\newtheorem{theorem}{Theorem}[section]

\newtheorem{definition}[theorem]{Definition}

\newtheorem{lemma} [theorem]{Lemma}

\title{\bf Weak Integer Additive Set-Indexers of Certain Graph Products}

\author{{\bf N K Sudev \footnote{Department of Mathematics, Vidya Academy of Science \& Technology, Thalakkottukara, Thrissur - 680501, Kerala, India. email: {\em sudevnk@gmail.com}}} and {\bf K A Germina\footnote{Department of Mathematics, School of Mathematical \& Physical Sciences, Central University of Kerala, Kasaragod-671316, Kerala, India. email:{\em srgerminaka@gmail.com}}}}
\date{}
\begin{document}
\maketitle

\begin{abstract}
Let $\mathbb{N}_0$ be the set of all non-negative integers and $\mathcal{P}(\mathbb{N}_0)$ be its power set. An integer additive set-indexer (IASI) is defined as an injective function $f:V(G)\to \mathcal{P}(\mathbb{N}_0)$ such that the induced function $f^+:E(G) \to \mathcal{P}(\mathbb{N}_0)$ defined by $f^+ (uv) = f(u)+ f(v)$ is also injective, where $f(u)+f(v)$ is the sumset of $f(u)$ and $f(v)$. An IASI $f$ is said to be a weak IASI if $|f^+(uv)|=max(|f(u)|,|f(v)|)~\forall ~ uv\in E(G)$. In this paper, we study the admissibility of weak IASI by certain graph products of two weak IASI graphs.
\end{abstract}

\noindent \textbf{Key Words:} Integer additive set-indexers, mono-indexed elements of a graph, weak integer additive set-indexers, sparing number of a graph.
\newline
\textbf{AMS Subject Classification: 05C78}

\section{Introduction}
For all  terms and definitions, not defined specifically in this paper, we refer to \cite{FH} and \cite{HIS}. Unless mentioned otherwise, all graphs considered here are simple, finite and have no isolated vertices.

Let $\mathbb{N}_0$ denotes the set of all non-negative integers. For all $A, B \subseteq \mathbb{N}_0$, the {\em sum set} of these sets is denoted by  $A+B$ and is defined by $A + B = \{a+b: a \in A, b \in B\}$. If either $A$ or $B$ is countably infinite, then their sum set is also countably infinite. Hence, the sets we consider here are all finite sets of non-negative integers. The cardinality of a set $A$ is denoted by $|A|$. 

An {\em integer additive set-indexer} (IASI) is defined in \cite{GA} as an injective function $f:V(G)\to \mathcal{P}(\mathbb{N}_0)$ such that the induced function $f^+:E(G) \to \mathcal{P}(\mathbb{N}_0)$ defined by $f^+ (uv) = f(u)+ f(v)$ is also injective.

\begin{lemma}
\cite{GS1} If $f$ is an IASI on a graph $G$, then $max\,(|f(u)|,\,|f(v)|) \le f^+(uv) \le |f(u)||f(v)|, \forall ~ u,v\in  V(G)$.
\end{lemma}

In \cite{GS1}, an IASI $f$ is said to be a {\em weak IASI} if $|f^+(uv)|=max\,(|f(u)|,|f(v)|)$ for all $u,v\in V(G)$. A weak IASI $f$ is said to be {\em weakly uniform IASI} if $|f^+(uv)|=k$, for all $u,v\in V(G)$ and for some positive integer $k$.  A graph which admits a weak IASI may be called a {\em weak IASI graph}. It is to be noted that if $G$ is a weak IASI graph, then every edge of $G$ has at least one mono-indexed end vertex. The cardinality of the labeling set of an element (vertex or edge) of a graph $G$ is called the {\em set-indexing number} of that element. An element (a vertex or an edge) of graph which has the set-indexing number $1$ is called a {\em mono-indexed element} of that graph. The {\em sparing number} of a graph $G$ is defined to be the minimum number of mono-indexed edges required for $G$ to admit a weak IASI and is denoted by $\varphi(G)$. The following are some major results proved in \cite{GS3}.

\begin{theorem}\label{T-WSG}
\cite{GS3} A subgraph of weak IASI graph is also a weak IASI graph.
\end{theorem}

\begin{theorem}\label{T-WUC}
\cite{GS3} A graph $G$ admits a weak IASI if and only if $G$ is bipartite or it has at least one mono-indexed edge. 
\end{theorem}

\begin{theorem}\label{T-WUOC}
\cite{GS3} An odd cycle $C_n$ has a weak IASI if and only if it has at least one mono-indexed edge. 
\end{theorem}

\begin{theorem}\label{T-NME}
\cite{GS3} Let $C_n$ be a cycle of length $n$ which admits a weak IASI, for a positive integer $n$. Then, $C_n$ has an odd number of mono-indexed edges when it is an odd cycle and has even number of mono-indexed edges, when it is an even cycle. 
\end{theorem}

\begin{theorem}\label{T-WUG}
\cite{GS4} The graph $G_1\cup G_2$ admits a weak IASI if and only if both $G_1$ and $G_2$ are weak IASI graphs. 
\end{theorem}

\begin{theorem}\label{T-WKN}
\cite{GS3} The complete graph $K_n$ admits a weak IASI if and only if the minimum number of mono-indexed edges of $K_n$ is $\frac{1}{2}(n-1)(n-2)$.
\end{theorem}

In this paper, we call a set $B$ an integral multiple of another set $A$ if every element of $B$ is an integral multiple of the corresponding element of $A$.

\section{Fundamental Products of Weak IASI Graphs}

In different products of given graphs, we need to take several layers or copies of some or all given graphs and to establish the adjacency between them according to certain rules. Hence, it may not be possible to induce a weak IASI to a graph product from the weak IASIs of the given graphs. We have to define an IASI independently for a graph product. We say that two copies of a graph are {\em adjacent} in a graph product if there exist some edges between the vertices of those copies in the graph product. 
In this section, we discuss the admissibility of weak IASI by the three fundamental products of two weak IASI graphs. First, consider the Cartesian product of two graphs which is defined as follows.


Let $G_1(V_1,E_1)$ and $G_2(V_2,E_2)$ be two graphs.Then, the {\em Cartesian product} of $G_1$ and $G_2$, denoted by $G_1\Box G_2$, is the graph with vertex set $V_1\Box V_2$  defined as follows. Let $u=(u_1, u_2)$ and $v=(v_1,v_2)$ be two points in $V_1\Box V_2$. Then, $u$ and $v$ are adjacent in $G_1\Box G_2$ whenever [$u_1=v_1$ and $u_2$ is adjacent to $v_2$] or [$u_2=v_2$ and $u_1$ is adjacent to $v_1$]. If $|V_i|=n_i$ and $|E_i|=m_i$ for $i=1,2$, then $|V(G_1\Box G_2)|=n_1n_2$  and $i=1,2$ and $|E(G_1\Box G_2)|=n_1m_2+n_2m_1$.

The Cartesian product of two bipartite graphs is also a bipartite graph. Also, the Cartesian products $G_1\Box G_2$ and $G_2\Box G_1$ of two graphs $G_1$ and $G_2$, are isomorphic graphs.

\begin{theorem}
Let $G_1$ and $G_2$ be two weak IASI graphs. Then, the product $G_1\Box G_2$ also admits a weak IASI. 
\end{theorem} 
\begin{proof}
Let $G_1$ and $G_2$ be two weak IASI graphs on $m$ and $n$ vertices respectively. 
%
We can view the product $G_1\Box G_2$ as a union of $n_2$ copies of $G_1$ and a finite number of edges connecting the corresponding vertices of two copies $G_{1i}$ and $G_{1j}$ of $G_1$ according to the adjacency of the corresponding vertices $v_i$ and $v_j$ in $G_2$, where $1\le i\neq j\le n_2$.

Let $u_{ij}$ be the $i$-th vertex of $G_{1j}$, the $j$-the copy of $G_1$. For odd values of $j$, label the vertices of $G_{1j}$ in such a way that the corresponding vertices of $G_{1j}$ have the same type of set-labels as that of $G_1$. That is, for odd $j$, let $u_{ij}$ has singleton set-label (or non-singleton set-label) according as the corresponding vertex $u_i$ of $G_1$ has singleton set-label (or non-singleton set-label). 

Let $u_i$ be not an end vertex of a mono-indexed edge In $G_1$. Then, for even values of $j$, label the corresponding vertex $u_{ij}$ in such a way that $u_{ij}$ has a singleton set-label (or non-singleton set-label) according as the vertex $u_i$ of $G_1$ has non-singleton set-label (or singleton set-label). Also, label the vertices of $G_{1j}$ which are corresponding to the adjacent mono-indexed vertices of $G_1$, by singleton sets. This labeling is a weak IASI for the graph $G_1\Box G_2$.                
\end{proof}

If a graph $G$ is the Cartesian product of two graphs $G_1$ and $G_2$, then $G_1$ and $G_2$ are called the {\em factors} of $G$. A graph is said to be {\em prime} with respect to a given graph product if it is non-trivial and can not be represented as the product of two non trivial graphs. 

\begin{theorem}\label{T-WIASI-GPCP2}
Let $G$ is a non-prime graph which admits a weak IASI. Then, every factor of $G$ also admits a weak IASI.
\end{theorem}
\begin{proof}
Let $G$ be a non-prime weak IASI graph with a weak IASI $f$. If $G_1$ is a factor of $G$, then there is a subgraph $G_{1i}$ in $G$ which is isomorphic to $G_1$ such that $v_i$ is the  vertex of $G_1$ corresponding to the vertex $v$ of $G_1$. define a function $g:V(G_1) \to \mathcal{P}(\mathbb{N}_0)$ defined by $g(v)=f'(v_i)$ where $f'= f|_{G_{1i}}$, the restriction of $f$ to the subgraph $G_{1i}$. By Theorem \ref{T-WSG}, $f'$ is a weak IASI of $G_{1i}$ and hence $g$ is a weak IASI of $G_1$.
\end{proof}

Next, recall the definition of another graph product called the direct product of two graphs.

\begin{definition}{\rm 
\cite{HIS} The {\em directed product} of two graphs $G_1$ and $G_2$, is the graph whose vertex set is $V(G_1)\times V(G_2)$ and for which the vertices $(u,v)$ and $(u',v')$ are adjacent if $uu'\in E(G_1)$ and $vv'\in E(G_2)$. The direct product of $G_1$ and $G_2$ is denoted by $G_1\times G_2$.}
\end{definition}

Note that the direct product of two connected graphs can be a disconnected graph also. The direct product is also known as {\em tensor product} or {\em cardinal product} or {\em cross product} or {\em categorical product} or {\em Kronecker product}. The following theorem establishes the admissibility of weak IASI by the direct product of two weak IASI graphs.

\begin{theorem}
The direct product of two weak IASI graphs admits a weak IASI.
\end{theorem}
\begin{proof}
Let $G_1$ and $G_2$ be two weak IASI graphs on $n_1$ and $n_2$ vertices, $m_1$ and $m_2$ edges respectively. Let $V(G_1)=\{u_1,u_2,u_3,\ldots u_{n_1}\}$ and $V(G_2)=\{v_1,v_2,v_3,\ldots v_{n_2}\}$. Make $n_2$ copies of $V(G_1)$, denoted by $V_j=\{u_{ij}: 1\le j\le n_2\}$. Since the vertex $u_{ij}$ is adjacent to a vertex $u_{rs}$ if $u_i$ and $u_r$ are adjacent $G_1$ and $u_j$ and $u_s$ are adjacent in $G_1$, no vertices in the copy $V_j$ can be adjacent to each other in $G_1\times G_2$. Hence, define a function $f_j$ on the vertex set $V_j$ such that it assigns the set-labels to the vertices of $V_j$, which are integral multiples of the set-labels of the corresponding vertices of $G_1$. Clearly, no two adjacent edges in $G_1\times G_2$ have non-singleton set-labels. Therefore, this labeling is a weak IASI on $G_1\times G_2$.
\end{proof}




Next, recall the definition of the strong product of two graphs.


\begin{definition}\label{Defn-SP}
\cite{HIS} The {\em strong product} of two graphs $G_1$ and $G_2$ is the graph, denoted by $G_1\boxtimes G_2$, whose vertex set is $V(G_1) \times V(G_2)$ and for which the vertices $(u,v)$ and $(u',v')$ are adjacent if $[uu'\in E(G_1)~\text{and}~ v=v']$ or $[u=u' ~ \text{and}~vv'\in E(G_2)]$ or $[uu'\in E(G_1)$ and $vv'\in E(G_2)]$.
\end{definition}

From this definition, we understand that $E(G_1\boxtimes G_2)= E(G_1\Box G_2)\cup E(G_1\times G_2)$. Now, we prove the existence of weak IASI for the strong product of two weak IASI graphs in the following theorem.

\begin{theorem}
The strong product of two weak IASI graph admits a weak IASI.
\end{theorem}
\begin{proof}
Let $G_1$ and $G_2$ be two weak IASI graphs on $n_1$ and $n_2$ vertices with the corresponding weak IASIs $f_1$ and and $f_2$ respectively. Let $G=G_1\boxtimes G_2$. Then, $G$ can be viewed as follows.
Take $n_2$ copies of $G_1$, denoted by $G_{1i}$, for $1\le i \le n_2$. Let $u_{ij}$ be the $i$-th vertex of the $j$-th copy of $G_1$, where $1\le i \le n_1, 1\le j \le n_2$. If a copy $G_{1j}$ is adjacent to another copy $G_{1k}$ in $G$, then the vertex $u_{ij}$ will be adjacent to the vertices $u_{i,k}, u_{i+1, j}, u_{i-1,j}$, if they exist. 

Let $f:V(G)\to \mathcal{P}(\mathbb{N}_0)$ which label the vertices of $G$ in the following way. Label the corresponding vertices of the first copy $G_{1\,1}$ of $G_1$ by the same set-labels of the vertices of $G_1$. Now, by Definition \ref{Defn-SP}, a vertex of the copies of $G_1$ that are adjacent to $G_{1\,1}$ can have a non-singleton set label if and only if the corresponding vertex and its adjacent vertices in $G_{1\,1}$ are mono-indexed. Let $G_{1r}$ be the next copy of $G_1$ which is not adjacent to $G_{1\,1}$. Label the vertices of this copy by an integral multiple of the set-labels of the corresponding vertices of $G_1$ and label the vertices of adjacent copies of $G_{1r}$ such that no vertex of $G_{1r}$ has a non-singleton set-label unless the corresponding vertex and its adjacent vertices in $G_{1r}$ are mono-indexed. Proceed in this way until all the vertices in $G$ are set-labeled. Then, we have a set-labeling in which no two adjacent vertices of $G$ have non-singleton set-labels. Hence, $f$ is a weak IASI on $G=G_1\boxtimes G_2$. This completes the proof.
\end{proof}


\section{Other Products of Weak IASI Graphs}

In the previous section, we have discussed the admissibility of weak IASI by three fundamental products of weak IASI graphs. Now, we proceed to discuss the existence of weak IASI for certain other graph products.

Now, recall the definition  of lexicographic product of two graphs.

\begin{definition}
\cite{IK} The {\em lexicographic product} or {\em composition} of two graphs $G_1$ and $G_2$ is the graph, denoted by $G_1\circ G_2$, is the graph whose vertex set $V(G_1)\times V(G_2)$ and for two vertices $(u,v)$ and $(u',v')$ are adjacent if $[uu'\in E(G_1)]$ or $[u=u'~\text{and}~ vv'\in E(G_2)]$.
\end{definition}

Admissibility of weak IASI by the lexicographic product of two weak IASI graphs is established in the following theorem. 

\begin{theorem}
The lexicographic product of two weak IASI graph admits a weak IASI.
\end{theorem}
\begin{proof}
Let $G_1$ and $G_2$ be two weak IASI graphs on $n_1$ and $n_2$ vertices respectively. The composition of $G_1$ and $G_2$ can be viewed as follows. Take $n_1$ copies of $G_2$, denoted by $G_{2i};~1\le i \le n_1$. Every vertex of a copy $G_{2i}$ is adjacent to all vertices of another copy $G_{2j}$ in $G_1\circ G_2$ if the corresponding vertices $v_i$ and $V_j$ are adjacent in $G_1$.

Label the corresponding vertices of the first copy $G_{2\,1}$ of $G_2$ by the same set-labels of the vertices of $G_2$. Since every vertex of $G_{2\,1}$ is adjacent to all vertices of its adjacent copies, these vertices must be labeled by distinct singleton sets. Now, label the vertices of the next copy $G_{2r}$ of $G_2$ which is not adjacent to $G_{2\,1}$  by an integral multiple of the set-labels of the corresponding vertices of $G_2$ and label the vertices of the adjacent copies of $G_{2r}$ by singleton sets. Proceed in this way until all the vertices in $G$ are set-labeled. This set-labeling is a weak IASI for $G_1\circ G_2$.
\end{proof}

Next, the graph product we are going to discuss is the corona of two weak IASI graphs.


\begin{definition}\label{D-5.1}{\rm
\cite{FH} By {\em corona} of two graphs $G_1$ and $G_2$, denoted by $G_1\odot G_2$, is the graph obtained by taking one copy of $G_1$ (which has $p_1$ vertices) and $p_1$ copies of $G_2$ and then joining the $i$-th point of $G_1$ to every point in the $i$-th copy of $G_2$. The number of vertices and edges in $G_1\odot G_2$ are $p_1(1+p_2)$ and $q_1+p_1q_2+p_1p_2$ respectively, where $p_i$ and $q_i$ are the number of vertices and edges of the graph $G_i, i=1,2$.}
\end{definition}

The following theorem establishes a necessary and sufficient condition for the corona of two  weak IASI graphs to admit a weak IASI.

\begin{theorem}\label{T-NMIEG}
Let $G_1$ and $G_2$ be two weak IASI graphs on $m$ and $n$ vertices respectively. Then,
\begin{enumerate}
\item[(i)] $G_1\odot G_2$ admits a weak IASI if and only if either $G_1$ is $1$-uniform or it has $r$ copies of $G_2$ that are 1-uniform, where $r$ is the number of vertices in $G_1$ that are not mono-indexed. 
\item[(ii)] $G_2\odot G_1$ admits a weak IASI if and only if either $G_2$ is $1$-uniform or it has $l$ copies of $G_2$ that are $1$-uniform, where $l$ is the number of vertices in $G_2$ that are not mono-indexed.
\end{enumerate} 
\end{theorem}
\begin{proof}

Consider the corona $G_1\odot G_2$. Let $f$ be a weak IASI of $G_1$ and $f_i$ be a weak IASI on the $i$-th copy $G_{2i}$ of $G_2$ all whose vertices are connected to the $i$-th vertex of $G_1$.  Define the function $g$ on $G_1\odot G_2$ by 
\[ g(v)= \left\{
\begin{array}{l l}
 	f(v) & \quad \text{if $v\in G_1$}\\
    f_i(v) & \quad \text{if $v\in G_{2i}$},~ 1\le i \le m
 \end{array} \right.\] 

Assume that $G_1\odot G_2$ is a weak IASI graph. If $G_1$ is $1$-uniform, then the proof is complete. If $G_1$ is not $1$-uniform, then the vertex set $V$ of $G_1$ can be divided into two disjoint sets $V_1$ and $V_2$, where $V_1$ is the set of all mono-indexed vertices in $G_1$ and $V_2$ be the set of all vertices that are not mono-indexed in $G_1$. Then, we observe that any copy $G_{2i}$ of $G_2$ that are connected to the vertices of $V_2$ cannot have a vertex that is not mono-indexed. That is, $r$ copies of $G_2$ are $1$-uniform, where $r=|V_2|$. Hence, $G_1\odot G_2$ is a weak IASI graph implies $G_1$ is $1$-uniform or $r$ copies of $G_2$ are $1$-uniform, where $r$ is the number of vertices of $G_1$ that are not mono-indexed.

Conversely, either $G_1$ or $r$ copies of $G_2$ that are $1$-uniform, where $r$ is the number of vertices of $G_1$ that are not mono-indexed. If $G_1$ is $1$-uniform, then the vertices of $G_{2i}$ can be labeled alternately by distinct singleton sets and distinct non-singleton sets under $f_i$. If $G_1$ is not $1$-uniform, by hypothesis, $r$ copies of $G_2$ are $1$-uniform, where $r$ is the number of vertices of $G_1$ that are not mono-indexed. Label the vertices of path $G_{2i}$, which is adjacent to the vertex $v_i$ of $G_1$ that are not mono-indexed, by distinct singleton sets under $f_i$.  Hence, in both cases, $g$ is a set-indexer and hence a weak IASI for $G_1\odot G_2$.

\noindent The proof for the second part is similar.
\end{proof}

We now proceed to determine the sparing number of the corona of two graphs.

\begin{theorem}
Let $G_1$ and $G_2$ be two weak IASI graphs on $n_1$ and $n_2$ vertices, $m_1$ and $m_2$ edges and $r_1$ and $r_2$ mono-indexed vertices respectively. Then, the sapring number of $G_1\odot G_2$ is $r_1(1+r_2)+(n_1-r_1)m_2$ and the sparing number of $G_2\odot G_1$ is $r_2(1+r_1)+(n_2-r_2)m_1$.
\end{theorem}
\begin{proof}
Since $G_1$ has $r_1$ mono-indexed vertices, $(n_1-r_1)$ copies of $G_2$ must be $1$-uniform in $G_1\odot G_2$. In the remaining $r_1$ copies, label the vertices by the set-labels which are some integral multiples of the  set-labels of  the corresponding vertices of $G_2$ (in such a way that no two copies of $G_2$ have the same set of set-labels). Hence, each of these copies contains the same number of mono-indexed edges as that of $G_2$. Therefore, the total number of mono-indexed edges in $G_1\odot G_2$ is $r_1+(n_1-r_1)m_2+r_1r_2=r_1(1+r_2)+(n_1-r_1)m_2$.

\noindent Similarly, we can prove the other part also.
\end{proof}


Another interesting graph product is the root product of two graphs. Recall the definition of rooted product of two weak IASI graphs.

\begin{definition}{\rm 
\cite{GM} The {\em rooted product} of a graph $G_1$ on $n_1$ vertices and rooted graph $G_2$ on $n_2$ vertices, denoted by $G_1\circ G_2$, is defined as the graph obtained by taking $n_1$ copies of $G_2$, and for every vertex $v_i$ of $G_1$ and identifying $v_i$ with the root node of the $i$-th copy of $G_2$. }
\end{definition}

The following theorem verifies the admissibility of weak IASI by the rooted product of two graphs.

\begin{theorem}
The rooted product of two weak IASI graphs is also a weak IASI graph.
\end{theorem}
\begin{proof}
Let $G_1$ and $G_2$ be the given graphs with the weak IASIs $f_1$ and $f_2$ defined on them respectively. Also let $V(G_1)=\{u_1,u_2,u_3,\ldots, u_{n_1}\}$ be the vertex set of $G_1$ and let $V(G_2)=\{v_1,v_2,v_3,\ldots, v_{n_2}\}$ be the vertex set of $G_2$. Let $G=G_1\circ G_2$. Without loss of generality, let $v_1$ be the root vertex of $G_2$. Now, make $n_1$ copies of $G_2$, denoted by $G_{2r}, 1\le r \le n_1$,  with $V(G_{2r}=\{v_{1r},v_{2r},v_{3r},\ldots, v_{{n_2}r}\}$.

Define a function $f:V(G)\to \mathcal{P}(\mathbb{N}_0)$ with the following conditions.

\begin{enumerate}
\item For $1\le i \le n_1$, define a function $f_{2r}:V(G_{2r})\to \mathcal{P}(\mathbb{N}_0)$ such that $f_{2r}(v_{ir})= r*f_2(v_i)$, where $r*f_2(v_i)$ is the set obtained by multiplying the elements of the set-label $f_2(v_i)$ by the integer $r$.
\item The vertex $u_r'$ obtained by identifying the vertex $u_r$ of $G_1$ and the root vertex $v_{1r}$ of the $r$-th copy $G_{2r}$ of $G_2$ has the same set-label of $u_r$ unless $u_r$ has a non-singleton set label and $v_{1r}$ is mono-indexed. In this case , let $u_r'$ assumes the same set-label of $v_{1r}$. 
\end{enumerate}
Then, under $f$, no two adjacent vertices of $G$ have non-singleton set-labels. That is, $f$ is a weak IASI on $G=G_1\circ G_2$. This completes the proof.
\end{proof}

\section{Conclusion}

In this paper, we have discussed the admissibility of weak IASI by the certain products of two graphs which admit weak IASIs. Some problems in this area are still open, as we have not studied about the sparing number of the graph product (except corona) of two arbitrary graphs $G_1$ and $G_2$, in our present discussion. Uncertainty in the adjacency pattern of different graph makes this study complex. An investigation to determine the sparing number of different products of two arbitrary graphs in terms of their orders, sizes and the vertex degrees in each of them, seem to be fruitful. The admissibility of weak IASIs by certain other graph products is also worth studying. 



\begin{thebibliography}{25}
\bibitem {BM1} J A Bondy and U S R Murty, (2008). {\bf Graph Theory}, Springer.
\bibitem {CZ} G Chartrand and P Zhang, (2005). {\bf Introduction to Graph Theory}, McGraw-Hill Inc.
\bibitem {FFH} R Frucht and F Harary (1970). {\em On the Corona of Two Graphs}, Aequationes Math., {\bf 4}(3), 322-325.
\bibitem {JAG1} J A Gallian, (2011). {\em A Dynamic Survey of Graph Labelling}, The Electronic Journal of Combinatorics (DS 16).
\bibitem {GA} K A Germina and T M K Anandavally, (2012). {\em Integer Additive Set-Indexers of a Graph:Sum Square Graphs}, Journal of Combinatorics, Information and System Sciences, {\bf 37}(2-4), 345-358.
\bibitem {GS1} K A Germina and N K Sudev, (2013). {\em On Weakly Uniform Integer Additive Set-Indexers of Graphs}, Int. Math. Forum, {\bf 8}(37), 1827-1834.
\bibitem {GM} C D Godsil and B D McKay, (1978). {\em A New Graph Product and its Spectrum}, Bull. Austral. Mat. Soc., {\bf 18}, 21-28. 
\bibitem {HIS} R Hammack, W Imrich and S Klavzar (2011). {\bf Handbook of Product graphs}, CRC Press.
\bibitem {FH}  F Harary, (1994). {\bf Graph Theory}, Addison-Wesley Publishing Company Inc.
\bibitem {IK} W Imrich, S Klavzar, (2000). {\bf Product Graphs: Structure and Recognition}, Wiley.
\bibitem {GS3} N K Sudev and K A Germina, (2014). {\em A Characterisation of Weak Integer Additive Set-Indexers of Graphs}, ISPACS J. Fuzzy Set Valued Analysis, {\bf 2014}, Article Id: jfsva-0189, 7 pages.
\bibitem {GS4} N K Sudev and K A Germina, (2014). {\em Weak Integer Additive Set-Indexers of Graph Operations}, Global J. Math. Sciences: Theory and Practical, {\bf 6}(1),25-36.
\bibitem{GS8} N K Sudev and K A Germina, (2014). {\em Weak Integer Additive Set-Indexers of Certain Graph Structures}, Annals of Pure and Appl. Math., {\bf 6}(2),140-149.
\bibitem{TRA} M Tavkoli, F Rhbarnia and A R Ashrafi, (2013). {\em Note on Strong Product of Graphs}, Kragujevac J. Math., 37(1), 187–193.
\bibitem{PMW} P M Weichsel, (1962). {\em The Kronecker Product of Graphs}, Proc of the Amer. Math. Soc. {\bf 13}(1): 47–52, doi:10.2307/2033769.
\bibitem {DBW} D B West, (2001). {\bf Introduction to Graph Theory}, Pearson Education Inc.
\end{thebibliography}
\end{document}